\documentclass[a4paper,12pt]{article}
\usepackage[english]{babel}
\usepackage{amsthm,amsmath,amssymb,color}
\usepackage{hyperref}
\newtheorem{rema}{Remark}
\newtheorem{lemma}{Lemma}
\newtheorem{corollary}{Corollary}
\newtheorem{prop}{Proposition}
\newtheorem{thm}{Theorem}
\newtheorem{defn}{Definition}

\def \a   {\alpha}   
\def \b   {\beta}
\def \g   {\gamma}
\def \d   {\delta}
\def \l   {{\lambda}}
\def \dd {\mathrm{d}}
\def \eps {\varepsilon}

\def\E{{\mathbb{E}}}

\def\P{{\mathbb{P}}}
\def\R {{\mathbb{R}}}
\newcommand{\Z}{\mathbb{Z}}
\def\F{{\cal{F}}}
\def\G{{\cal{G}}}

\def\limt{\lim_{t\to\infty}}
\def\limt0{\lim_{t\to 0}}

\def\|{\,|\,}
\renewcommand{\mod}{\ \mathrm{mod}\ }
\def\sb{\mathcal{R}ademacher}

\title{Recurrence and transience of Rademacher series}
\author{Satyaki Bhattacharya and Stanislav Volkov\footnote{Centre for Mathematical Sciences, Lund University, Box 118 SE-22100, Lund, Sweden}} 

\begin{document}
\maketitle
\begin{abstract}
We introduce the notion of {\bf a}-walk $S(n)=a_1 X_1+\dots+a_n X_n$, based on a sequence of positive numbers ${\bf a}=(a_1,a_2,\dots)$ and a Rademacher sequence $X_1,X_2,\dots$. We study recurrence/transience (properly defined) of such walks for various sequences of ${\bf a}$. In particular, we establish the classification  in the cases where $a_k=\lfloor k^\b\rfloor$, $\b>0$,
as well as in the case $a_k=\lceil \log_\g k \rceil$ or $a_k=\log_\g k$ for $\g>1$.
\end{abstract}



\section{Introduction}
We will say that a random variable $X$ has a Rademacher distribution and write $X\sim \sb$, if $\P(X=+1)=\P(X=-1)=\frac 12$.

Let $X_i\sim \sb$, $i=1,2,\dots$, be i.i.d., and $\F_n =\sigma(X_1,X_2,\dots,X_n)$ be the sigma-algebra generated by the first $n$ members of this sequence. Let ${\bf a}=(a_1,a_2,\dots)$ be a non-random sequence of positive numbers. Define the ${\bf a}$-walk as 
$$
S(n)=a_1 X_1+a_2 X_2+\dots +a_n X_n=\sum_{k=1}^n a_k X_k
$$
with the convention $S(0)=0$. 
\begin{defn}\label{def1}
Let $C\ge 0$. We call the ${\bf a}$-walk $S$ defined above {\em $C-$recurrent}, if the event $\{|S(n)|\le C\}$ occurs for infinitely many $n$. (In case when $C=0$, this is equivalent to the usual recurrence, i.e., $S(n)=0$ for infinitely many $n$, so we will call the walk just {\em recurrent}.)

We call the ${\bf a}$-walk  {\em transient}, if it is not $C$-recurrent for any $C\ge 0$.
\end{defn}

Our aim is to determine the probability  that the ${\bf a}$-walk for given ${\bf a}$ and $C$ is recurrent; in principle, this probability may be different\footnote{For example, if ${\bf a}=(1,1,3,3,3,3\dots)$ then the {\bf a}-walk is recurrent with probability $1/2$.} from $0$ and $1$. A simplest example of an ${\bf a}$-walk is when all $a_i\equiv a\in\R_+$. Such a random walk is obviously a.s.\ recurrent since it is equivalent to the one-dimensional simple random walk.

The question of recurrence is naturally related to the Littlewood-Offord problem which deals with the maximization of probability $\P(S(n)=v)$ over all $v$, subject to various hypotheses on ${\bf a}$.
In particular, in~\cite{TAO} the authors develop an inverse Littlewood-Offord theory, using which they show that this probability is large only when the elements of ${\bf a}$ are  contained in a generalized arithmetic progression; see also~\cite{NG}.

The study of ${\bf a}$-walk is also somewhat relevant to the conjecture by Boguslaw Tomaszewski (1986), which says that  $\P\left( |S(n)| \le \sqrt{a_1^2+\dots+a_n^2}\right) \ge \frac12$ for all sequences ${\bf a}$ and all $n$. The conjecture was recently proved by Nathan Keller and Ohad Klein in~\cite{Klein}.

Let us first start with some general statements.
First, we show that the choice of $C>0$ is sometimes unimportant for the definition of $C$-recurrence.
\begin{thm}
Suppose that $a_n\to\infty$ and at the same time $|a_n-a_{n-1}|\to 0$ as $n\to\infty$. Then if an ${\bf a}$-walk is $C$-recurrent with a {\em positive} probability for some $C>0$ then it is $\tilde C$-recurrent with a {\em positive} probability for {\em all} $\tilde C>0$.
\end{thm}
\begin{proof}
Since the notion of $C$-recurrence is monotone in $C$, i.e.\, if an ${\bf a}$-walk is $C_1-$recurrent for $C_1>0$ then it is $C_2$-recurrent for all $C_2\ge C_1$, it suffices to prove that $C-$recurrence implies $\frac{2C}3-$recurrence.

Indeed, suppose the ${\bf a}$-walk is $C$-recurrent; formally, if we define the events
\begin{align*}
E&=\{S(n)\in [-C,C]\text{ for infinitely many }n\},\\
\tilde{E}&=\{S(n)\in [-2C/3,2C/3]\text{ for infinitely many }n\}
\end{align*}
then $\P(E)>0$. We want to show that $\P(\tilde{E})>0$ as well.

Let $n_1$ be so large that $|a_i-a_{i-1}|<C/6$ for all $i\ge n_1$. Define the sequence $n_k$, $k\ge 2$, by setting
$$
n_k=\min\{i\ge n_{k-1}+1:\ a_i\ge a_{n_{k-1}}+C/6\}
$$ 
(which is well-defined since $a_i\to\infty$), then trivially
\begin{align}\label{eqNML}
 \frac{C}{6} \le a_{n_{k+1}}-a_{n_k}\le \frac{C}{3}\qquad \text{for each }k=1,2,\dots
\end{align}

Fix a positive integer $K$ and for $y=(y_1,y_2,\dots,y_K)\in\Omega_K:=\{-1,+1\}^K$ define 
\begin{align*}
\bar{X}_K&= \{X_1,X_2,\dots,X_K\};\\
s_y&=a_1 y_1+a_2 y_2+\dots a_k y_K.
\end{align*}
Let $y\in\Omega_K$ be such that $\P(\{\bar{X}_K=y\}\cap E)>0$. Observe that
\begin{align*}
\{\bar{X}_K=y\}\cap E  &=\{\bar{X}_K=y\}\cap B_K(s_y)
\end{align*}
where
\begin{align*}
B_K^+(u)&=
\{\text{there exist}\ m_1<m_2<\dots
\text{ such that }u+\sum_{i=K+1}^{m_j} a_i X_i\in [0,C]\};
\\
B_K^-(u)&=
\{\text{there exist}\ m_1'<m_2'<\dots
\text{ such that }u+\sum_{i=K+1}^{m_j'} a_i X_i\in [-C,0]\};
\\
B_K(u)&=B_K(u)^+\cup B_K(u)^-.
\end{align*}
Since $\{\bar{X}_K=y\}$ and $B_K(u)$ are independent, we have
$$
\P(\{\bar{X}_K=y\}\cap B_K(s_y))=\P(\bar{X}_K=y)\, \P(B_K(s_y)).
$$
Consequently, $\P(B_K(s_y))>0$, and as a result, $\P(B_K^+(s_y))>0$ or $\P(B_K^-(s_y))>0$ (or both).

Let $\Omega_K^*\subseteq \Omega_K$ contain those $y$s for which there is an index $k$ such that $n_{k+2}\le K$ and $y_{n_k}=-1$, $y_{n_{k+1}}=+1$, $y_{n_{k+2}}=-1$; let $k$ be the smallest such index.
For  $y\in\Omega_K^*$ define the mappings  $\sigma^+,\sigma^-:\Omega_K^*\to\Omega_K$ by
\begin{align*}
\sigma^+(y)&=
\begin{cases}
-y_i, &\text{if }i=n_k\text{ or }i=n_{k+1};\\
y_i,&\text{otherwise};
\end{cases}
\\
\sigma^-(y)&=
\begin{cases}
-y_i, &\text{if }i=n_{k+1}\text{ or }i=n_{k+2};\\
y_i,&\text{otherwise}.
\end{cases}
\end{align*}
Then for $y\in\Omega^*_K$
\begin{align*}
s_{\sigma^+(y)}&=s_y+2a_{n_k}-2a_{n_k+1}\in [s_y-2C/3,s_y-C/3],\\
s_{\sigma^-(y)}&=s_y-2a_{n_k}+2a_{n_k+1}\in [s_y+C/3,s_y+2C/3].
\end{align*}
As a result, it is not hard to see that
\begin{align*}
\{\bar{X}_K=\sigma^+(y)\}\cap B_K^+(s_y)
&  \subseteq
\left\{\sum_{i=1}^{m} a_i X_i\in [-2C/3,2C/3]\text{ for infinitely many }ms\right\}=\tilde{E},\\
 \{\bar{X}_K=\sigma^-(y)\}\cap B_K^-(s_y)
&  \subseteq
\left\{\sum_{i=1}^{m} a_i X_i\in [-2C/3,2C/3]\text{ for infinitely many }ms\right\}=\tilde{E}.
\end{align*}
Since at least one of $B_K^+(s_y)$ and $B_K^-(s_y)$ has a positive probability, $\P\left(\bar{X}_K=\sigma^\pm(y)\right)=2^{-K}$ and  the events on the LHS are independent, we conclude that $\P(\tilde E)>0$.

Now it only remains to show that there exists $y\in \Omega_K^*$ such that $\P(\{\bar{X}_K=y\}\cap E)>0$. Let 
$\kappa:=\kappa(K)=\max\{k\in\Z_+:\ n_k\le K\}$; obviously, $\kappa(K)\to\infty$ as $K\to\infty$. If we choose $y$ from $\Omega_K$ uniformly, we can trivially bound the probability that $y\notin \Omega_K^*$ by\footnote{exact: see the sequence {\rm A005251} in the online encyclopedia of integer sequences (\rm{https://oeis.org/A005251}), $\P\left(\bar{X}_K\notin \Omega^*_K\right)\approx \l^\kappa$, 
$\l={\frac {\sqrt [3]{100+12\,\sqrt {69}}}{6}}+{\frac {2}{3\,\sqrt [3]{100
+12\,\sqrt {69}}}}+{\frac{2}{3}}=0.877...$,
$\kappa=\kappa(K)$}
$$
\left(1-\frac 18\right)^{\lfloor \kappa/3\rfloor}\to 0\quad\text{as }\kappa\to\infty
$$
by grouping together triples $\left(X_{n_1},X_{n_2},X_{n_3}\right)$, $\left(X_{n_4},X_{n_5},X_{n_6}\right)$, etc.; in each such a triple 
$$
\P\left( (X_{n_k},X_{n_{k+1}},X_{n_{k+2}})=(-1,+1,-1)\right)=1/8.
$$
Hence
\begin{align}\label{above}
\P(E)=\sum_{y\in \Omega_K}\P(\{\bar{X}_K=y\}\cap E)
= \sum_{y\in \Omega_K^*}\P(\{\bar{X}_K=y\}\cap E)
+\P\left(\left\{\bar{X}_K\in \Omega_K\setminus \Omega_K^* \right\}\cap E\right).
\end{align}
Since $\P\left(\left\{\bar{X}_K\in \Omega_K\setminus \Omega_K^* \right\}\cap E\right)\le  \P\left(\bar{X}_K\in \Omega_K\setminus \Omega_K^* \right)$, by making $K$ sufficiently large, we can ensure that the second term on the RHS of~\eqref{above} is less than $\P(E)$, implying that there exist some $y\in\Omega_K^*$ such that $\P(\{\bar{X}_K=y\}\cap E)>0$, as required.
\end{proof}

\vskip 5mm

Our next result shows that if the sequence ${\bf a}$ is non-decreasing, then the walk will always ``jump'' over $0$ infinitely many times, even if the walk is not $C$-recurrent.
\begin{thm}\label{th:saty}
Suppose that $a_i$ is a non-decreasing positive sequence. Then the event $\{S(n)>0\}$ holds for infinitely many $n$ a.s.
The same is true for the event $\{S(n)<0\}$.
\end{thm}
The theorem immediately follows from symmetry and the more general
\begin{prop}\label{prop:Saty}
Suppose that $a_i$ is a non-decreasing sequence, $m$ is an integer such that $a_{m+1}>0$, and  $S(m)=A>0$.
Define
$$
\tau=\inf\{k\ge 0:\ S(m+k)\le 0\}.  
$$
Let $Y_j\sim\sb$ be  i.i.d., and 
$$
\tilde\tau=\inf\{k\ge 0:\ Y_1+Y_2+\dots+Y_k\le -r\}
$$
where $r=\lceil A/a_{m+1}\rceil$; note that $\tilde\tau<\infty$ a.s.\ and that, in fact, $Y_1+\dots+Y_{\tilde\tau}=-r$. 
Then $\tau$ is stochastically smaller than $\tilde\tau$, that is,
$$
\P(\tau>m)\le\P(\tilde\tau>m), \ \ m=0,1,2,\dots
$$
\end{prop}
\begin{proof} We will use coupling. Indeed, we can write
$$
S(m+j)=A+a_{m+1} Y_1+a_{m+2} Y_2+
\dots +a_{m+j} Y_j,\quad j=1,2,\dots
$$
Suppose that $\tilde\tau=k$, that is
\begin{align*}
Y_1>-r,\ Y_1+Y_2>-r,\ \dots,\  Y_1+Y_2+\dots+Y_{k-1}&>-r;
\\
Y_1+Y_2+\dots+Y_{k-1}+Y_k&= -r.
\end{align*}
Then, recalling that $a_i$ is a non-decreasing sequence,
\begin{align*}
S(m+k)&=A+a_{m+1}Y_{1}+\dots+a_{m+k-1} Y_{k-1}+a_{m+k} Y_k
\\&
\le A+a_{m+1}Y_{1}+\dots+a_{m+k-2} Y_{k-2}+a_{m+k-1} Y_{k-1}+a_{m+k-1} Y_k \quad(\text{since }Y_k=-1)\\
&=A+a_{m+1}Y_{1}+\dots+a_{m+k-2} Y_{k-2}+ a_{m+k-1} [Y_{k-1}+Y_k]
\\ &
\le A+a_{m+1}Y_{1}+\dots+a_{m+k-2} Y_{k-2}+ a_{m+k-2} [Y_{k-1}+Y_k]
\\ &=
A+a_{m+1}Y_{1}+\dots+a_{m+k-2} [Y_{k-2}+Y_{k-1}+ Y_k]
\\ &
\le \dots
\le A+a_{m+1}[Y_1+{\scriptstyle \dots}+Y_k]
=A-r a_{m+1} \le 0,
\end{align*}
since $Y_k$, $Y_{k-1}+Y_k$, $Y_{k-2}+Y_{k-1}+Y_k$, $\dots$, $Y_1+\dots+Y_k$ are all negative. Therefore, $\tau\le \tilde \tau$.
\end{proof}

Throughout the paper we will use a version of the Azuma-Hoeffding inequality; compare with the results of~\cite{MONT}.
\begin{lemma}\label{lem:AH}
Suppose that $b_1,b_2,\dots,b_m$ is a sequence of non-negative numbers and $\mathcal{S}=b_1 Y_1+b_2 Y_2+\dots+b_m Y_m$, where $Y_j\sim \sb$ are i.i.d. Then 
\begin{align}\label{eq:Azcont}
\P\left(|\mathcal{S} | \ge A \right)\le 2\exp\left(-\frac{A^2}{2(b_1^2+\dots+b_m^2)}\right)\text{ for all }A>0.
\end{align}
\end{lemma}

We also state the following fairly standard result.
\begin{lemma}\label{lem:CLT}
Let $T_i=Y_1+\dots+Y_i$ be a simple random walk. Suppose that $L_k$ and $y_k$, $k=1,2,\dots$, are two sequences such that $L_k\to\infty$, $y_k\to\infty$ and $y_k/\sqrt{L_k}\to r>0$.
Then
$$
\lim_{k\to\infty} \P\left(\max_{1\le i\le L_k} T_i\ge y_k\right)= 2\,\P(\eta\ge r)=2-2\,\Phi(r)
$$
where $\eta\sim\mathcal{N}(0,1)$ and $\Phi(\cdot)$ is its CDF.
\end{lemma}
\begin{proof}
Let $\tilde y_k=\lceil y_k\rceil\in\Z_+$. By the reflection principle, 
\begin{align*}
\P\left(\max_{1\le i\le L_k} T_i\ge y_k\right)
&=\P\left(\max_{1\le i\le L_k} T_i\ge \tilde y_k\right)
=2\,\P( T_{L_k}\ge \tilde y_k)-\P( T_{L_k}= \tilde y_k)
\\ &
=2\,\P\left( \frac{T_{L_k}}{\sqrt{L_k}}\ge \frac{\tilde y_k}{\sqrt{L_k}}\right)+O\left(\frac1{\sqrt{L_k}}\right)
\to 2\P(\eta\ge r)
\end{align*}
by the Central Limit Theorem, using also the fact that $\tilde y_k/y_k\to 1$.
\end{proof}

\section{Integer-valued {\bf a}-walks}
Suppose that the sequence ${\bf a}$ contains only integers.
\begin{prop}\label{prop:fourier}
Let $z\in\Z$. Suppose that the sequence
$$
\int_0^\pi \cos(tz)\, \prod_{k=1}^n \cos(t a_k)\dd t,\qquad n=1,2,\dots
$$
is summable. Then the events $\{S(n)=z\}$ occur for finitely many $n$ a.s.
\end{prop}
\begin{proof}
The result follows from standard Fourier analysis. Indeed,
$$
\E e^{it S(n)}=\sum_{k\in\Z} e^{itk} \P(S(n)=k)
$$
where the sum above goes, in fact, effectively over a finite number of $k$s (as $|S(n)|\le a_1+\dots+a_n$). At the same time,
$$
\int_{-\pi}^\pi  e^{it(k-z)}\dd t=\begin{cases}
2\pi, &\text{if } k=z;\\
0, &\text{if } k\in\Z\setminus\{z\}.
\end{cases}
$$
By changing the order of summation and integration, we obtain
$$
\frac 1{2\pi}\int_{-\pi}^\pi  \E e^{it (S(n)-z)}\dd t
=\frac1{2\pi}\sum_{k\in\Z} \int_{-\pi}^\pi e^{it(k-z)}\, \P(S(n)=k)\, \dd t=\P(S(n)=z).
$$
On the other hand,
$$
\frac 1{2\pi}\int_{-\pi}^\pi   \E e^{it (S(n)-z)}\dd t=
\frac 1{2\pi}\int_{-\pi}^\pi  e^{-itz}\,\prod_{k=1}^n \E e^{it a_k X_k}\dd t=\frac 1{\pi}\int_0^\pi \cos(tz)\, \prod_{k=1}^n \cos(t a_k)\dd t
$$
by the symmetry of $\cos$ and the fact that the imaginary part must equal zero. Now the result follows from the Borel-Cantelli lemma, since $\sum_n \P(S(n)=z)<\infty$.
\end{proof}
\begin{corollary}
Suppose that the sequence
$$
\int_0^\pi \left|\prod_{k=1}^n \cos(t a_k)\right|\dd t,\qquad n=1,2,\dots
$$
is summable. Then the {\bf a}-walk is transient a.s.
\end{corollary}
\begin{proof}
From Proposition~\ref{prop:fourier} we know that for each $z\in\Z$
$$
\pi \P(S(n)=z)=\int_0^\pi \cos(tz)\, \prod_{k=1}^n \cos(t a_k)\dd t\le 
\int_0^\pi \left|\cos(tz)\, \prod_{k=1}^n \cos(t a_k)\right|\dd t
\le \int_0^\pi \left|\prod_{k=1}^n \cos(t a_k)\right|\dd t.
$$
Hence the event $\{S(n)=z\}$ occurs finitely often a.s.\ for each $z$. Since for each $C>0$ there are only finitely many integers in $[-C,C]$ we conclude that the walk is not $C$-recurrent a.s.\ for every $C$.
\end{proof}

An interesting and quite natural example is when ${\bf a}=(1,2,3,\dots)$, i.e., $a_i=i$. It was previously published in the IMS Bulletin~\cite{IMS}, in the Student Puzzle Corner no.~37.
\begin{thm}\label{thm:puzzle}
The {\bf a}-walk with ${\bf a}=(1,2,3,\dots)$ is a.s.\ transient.
\end{thm} 
This statement follows from a much stronger Theorem~\ref{thm:SAR}, but for the sake of completeness, we present its self-contained proof.
\begin{proof}[Proof of Theorem~\ref{thm:puzzle}]
Let $A_n=\{S(n)=0\}=\{X_1+2X_2+\dots+nX_n=0\}.$ Then
$\P(A_n)=Q_n/2^n$, where 
$$
Q_n=\text{number of ways to put $\pm$ in the sequence $*1*2*3*\dots *n$ such that the sum equals $0$}.
$$
For example, $Q_1=Q_2=0$, $Q_3=Q_4=2$, $Q_5=Q_6=0$, $Q_7=8$, $Q_8=14$, etc.
It was essentially shown in~\cite{SUL} that 
$$
Q_n\sim \sqrt{\frac{6}{\pi}}\, \frac{2^n}{n^{3/2}}
\qquad\text{ when }n\mod 4\in\{0,3\}
$$
(and zero otherwise) as $n\to\infty$, meaning that the ratio of the RHS and the LHS converges to one.
Consequently, $\sum_n \P(A_n) \sim   \sum_n \frac{\rm const}{n^{3/2}}<\infty$ and the events $A_n$ occur a.s.\ finitely often by the Borel-Cantelli lemma. Hence the walk is a.s.\ not recurrent. 

Moreover, since for any $m\in\Z$
$$
\P(S(n+2|m|)=S(n)-m\mid \F_n)\ge  \frac{1}{2^{2m}}
$$
(by making the signs of $X_{n+1},X_{n+2},\dots,X_{n+2|m|}$ alternate), we conclude that if the event $\{S(n)=m\}$ occurs infinitely often, then $A_n$ shall also occur  infinitely often a.s., leading to contradiction. As a result,  $\P(\{S(n)=m\}\text{ i.o.})=0$ for all integer $m$s, and thus the walk is a.s.\ not $C$-recurrent for any non-negative $C$.
\end{proof}

\begin{rema} Though the $(1,2,3,\dots)$-walk is transient, it still can jump over zero infinitely many times, as it was shown by Theorem~\ref{th:saty}.
\end{rema}

In fact, Theorem~\ref{thm:puzzle} can be generalized greatly, using the  result from~\cite{SAR}, or even a weaker result of~\cite{ERD}, which provide the estimates for the maximum number of solutions of the equation $\sum_{i=1}^n \varepsilon_i a_i=t$ where $\epsilon_i\in\{0,1\}$ while $a_i$'s and $t$ are all integers.

\begin{thm}\label{thm:SAR}
Let ${\bf a}$ be such that all $a_i$'s are distinct integers. Then {\bf a}-walk  is a.s.\ transient.
\end{thm}
\begin{proof}
The main result of~\cite{SAR} implies that for any $\epsilon>0$
$$
\mathrm{card}(\{(x_1,x_2,\dots, x_n): \text{ all } x_i=\pm 1,\ 
a_1 x_1+\dots+a_n x_n=m\})\le \frac{(1+\epsilon)2^{n+3}}{n^{3/2}\sqrt{\pi}}
$$
for all $n\ge n_0(\epsilon)$ and all $m$. Setting $\epsilon=1$, and fixing $m\in\Z$, we obtain that
$$
\sum_{n=n_0(1)}^\infty \P(S(n)=m)\le 
\sum_{n=n_0(1)}^\infty
\frac{2\cdot 2^{n+3}}{n^{3/2}\sqrt{\pi}}\times \frac 1{2^n}
=\frac{16}{\sqrt\pi}\sum_{n=n_0(1)}^\infty
\frac{1}{n^{3/2}}
<\infty.
$$
Therefore, by the Borel-Cantelli lemma, only finitely many events $\{S(n)=m\}$ occur a.s. Since $S(n)$ takes only integer values, this implies that $\{|S(n)|\le C\}$ happens finitely often a.s.\ for any $C>0$.
\end{proof}
\begin{rema}
${}$\\ \noindent
(a)  It is not difficult to see that under the condition of Theorem~\ref{thm:SAR} it suffices that all $a_k$'s are distinct only starting from some $k_0\ge 1$.

\noindent
(b) If $a_k=\lfloor k^\b\rfloor$ with $\b\ge 1$, then we immediately have a.s.\ transience by Theorem~\ref{thm:SAR}.
\end{rema}

\section{A non-trivial recurrent example}
We assume here that ${\bf a}=(B_1,B_2,B_3,\dots)$ where each $B_k$ is a consecutive block of $k$'s of length precisely~$L_k\ge 1$ . Denote also by $i_k=1+L_1+L_2+\dots+L_{k-1}$ the index of the first element of the $k-$th block.
For example, if $L_k=2^k$, then  $i_k=2^k-1$ and
$$
{\bf a}=(1,1,\ \underbrace{2,2,2,2}_{L_2\text{ times}},\ \underbrace{3,3,3,3,3,3,3,3}_{L_3\text{ times}},\underbrace{4,\dots,4}_{L_4\text{ times}},\dots),
$$
one can also notice that $a_i=\lfloor \log_2 (i+1)\rfloor=\lceil \log_2 (i+2)\rceil-1$.
\begin{thm}\label{th:logint}
Suppose that for some $\eps>0$, $r>0$, and $ k_0$ we have
\begin{align}\label{tcond}
\begin{split}
\frac{L_k}{L_1+L_2+\dots+L_{k'}}&\ge (2+\eps)\ln k;\\
\frac{L_k}{L_{k'+1}+L_{k'+2}+\dots+L_{k-1}}&\ge 2r;\\
L_k\ge  k^4,
\end{split}
\end{align}
whenever $k-k'\ge \frac{k}{\ln k} -2 $ and $k,k'\ge k_0$.
Then the ${\bf a}$-walk described above is a.s.\ recurrent.
\end{thm}
\begin{rema}
One can easily check that the conditions of the theorem are satisfied if $a_k=\left\lfloor (\log_{\g} k)^\beta \right\rfloor$, where $\g>1$ and $\beta\in (0,1]$.
\end{rema}

\begin{proof}[Proof of Theorem~\ref{th:logint}] We will proceed in FIVE steps.
\vskip 1mm
\noindent\underline{\bf Step 1: Preliminaries}

First, we need the following lemma, which is probably known.

\begin{lemma}\label{lemLD}
Let $m\in\Z_+$ and $T_m$ be a simple symmetric random walk on $\Z^1$, that is, $T_m=Y_1+\dots+Y_m$,  where $Y_i\sim\sb$ are i.i.d. There exists a universal constant $c_1>0$ such that for all integers $z$ such that $|z|\le 2\sqrt{m}$, assuming that  $m$ is sufficiently large and $m+z$ is even,
$$
\P(T_m=z)\ge \frac{c_1}{\sqrt m}.
$$
\end{lemma}
\begin{proof}
 W.l.o.g.\ assume $z\ge 0$. We have
$$
\P(T_m=z)=\P\left(\frac{T_m+m}{2}=\frac{z+m}{2}\right)
=\P(\tilde T=w)
$$
where $\tilde T\sim {\rm Bin}(m,1/2)$ and $w=\frac{z+m}{2}\in\Z_+$. Note that $\tilde{m}\le w\le \tilde{m}+\sqrt m$ where $\tilde{m}=m/2$. So
\begin{align*}
\P(\tilde T=w)&=\binom{m}{w} 2^{-m}
=\binom{2\tilde{m}}{\tilde{m}}\frac 1{2^{2\tilde{m}}} \,
\frac{\tilde{m}!\, \tilde{m}!}{w!(m-w)!} 
=\frac{1+o(1)}{\sqrt{\pi \tilde{m}}}\,
\frac{(2\tilde{m}-w+1)(2\tilde{m}-w+2)\dots \tilde{m}}
{(\tilde{m}+1)(\tilde{m}+2)\dots w}
\\ &
=\frac{1+o(1)}{\sqrt{\pi \tilde{m}}}\,
\left(1-\frac{w-\tilde{m}}{\tilde{m}+1}\right)
\left(1-\frac{w-\tilde{m}}{\tilde{m}+2}\right)
\dots
\left(1-\frac{w-\tilde{m}}{w}\right)
\\ &
\ge 
\frac{1+o(1)}{\sqrt{\pi \tilde{m}}}\,
\left(1-\frac{\sqrt m}{\tilde{m}+1}\right)^{w-\tilde{m}}
\ge 
\frac{1+o(1)}{\sqrt{\pi \tilde{m}}}\,
\left(1-\frac{ \sqrt 2+o(1)}{\sqrt{\tilde{m}}}\right)^{\sqrt{2\tilde{m}}}
=\frac{e^{-2}+o(1)}{\sqrt{\pi m/2}}\ge \frac{0.1}{\sqrt{m}}
\end{align*}
for large enough $m$.
\end{proof}

\begin{corollary}\label{cordiv}
Let $T_m$, $m=0,1,2,\dots$, be as simple symmetric random walk as in Lemma~\ref{lemLD}.  Assume that $m$ and $k$ are positive integers such that  $k^2\le m$. Let $u\in \Z$, and either $k$ is odd, or both $k$ and $m-u$ are even.
Then for large $k$s
$$
\P(T_m-u \mod k=0)\ge \frac{c_1}{2k}
$$
where $c_1$ is the constant from Lemma~\ref{lemLD}.
\end{corollary}
\begin{proof}[Proof of Corollary~\ref{cordiv}]
First, assume that $m$, and hence $u$, are both even. Since $(T_m-u) \mod k=0 \Longleftrightarrow T_m=\tilde u \mod k$, where $\tilde u=(u\mod k)\in\{0,1,2,\dots,k-1\}$, it suffices to show the statement for $\tilde u$.

Let $M=\lfloor 2\sqrt{m}\rfloor\in ( 2\sqrt{m}-1, 2\sqrt{m}]$ and define
\begin{align*}
\mathbb{I}=[-M,-M+1,\dots,-1,0,1,\dots,M]=
\mathbb{I}_0\cup \mathbb{I}_1;
\\
\mathbb{I}_0=\{z\in\mathbb{I}:\text{ $z$ is even}\};
\quad
\mathbb{I}_1=\{z\in\mathbb{I}:\text{ $z$ is odd}\}.
\end{align*} 
There are at least $M$ elements in each $\mathbb{I}_0$ and $\mathbb{I}_1$.

\underline{If $k$ is odd}, then each of these two sets contains at least $\lfloor \frac{M}{k} \rfloor$ elements $z$ such that $z=\tilde u\mod k$. If $m$ is even (odd, resp.) for all $z$ either in $\mathbb{I}_0$ (in $\mathbb{I}_1$, resp.) by Lemma~\ref{lemLD} for large $k$s (and hence large $m$) we have 
$\P(T_m=z)\ge c_1/\sqrt{m}$. Consequently,
\begin{align*}
\P(T_m=\tilde u \mod k)&\ge \sum_{z\in \mathbb{I},\ z=\tilde u \mod k}
\P(T_m=z) \ge \left\lfloor \frac Mk\right\rfloor \times \frac{c_1}{\sqrt{m}}
\ge 
\left( \frac Mk -1\right) \times
\frac{c_1}{\sqrt{m}}
\\
&\ge 
\left( \frac {2\sqrt{m}-1}k -1\right) \times
\frac{c_1}{\sqrt{m}}
\ge
\left(1-\frac1k\right) \times
\frac{c_1}{\sqrt{m}}
=\frac{c_1}{k}-O(k^{-2})
\end{align*}
since $m\ge k^2$.
 
\underline{If $k$ is even}, then if $m$ is even (and thus $a$ is also even) then $\mathbb{I}_0$ contains at least $\lfloor \frac{M}{k} \rfloor$ elements $z$ such that $z=\tilde u\mod k$ and at the same time Lemma~\ref{lemLD} is applicable for $z\in\mathbb{I}_0$. On the other hand,  
if $m$ (and so $u$) is odd then $\mathbb{I}_1$ contains at least $\lfloor \frac{M}{k} \rfloor$ elements $z$ 
such that $z=\tilde u\mod k$ and Lemma~\ref{lemLD} is applicable for $z\in\mathbb{I}_1$. The rest of the proof is the same as for the case when $k$ is odd.
\end{proof}

\vskip 5mm
\noindent\underline{\bf Step 2: Splitting $S(n)$}

Recall that that $i_k$ denotes the first index of block $k$
and note that the sum of all the steps within block~$k$ can be represented as 
$$
S(i_{k+1}-1)-S(i_k-1)=k\cdot T_k,\qquad
T_k=X_1^{(k)}+\dots+ X_{L_k}^{(k)}
$$ 
where $X_j^{(k)}$'s are i.i.d.\ $\sb$ random variables.

For $m=2,\dots$, let
\begin{align}\label{defkm}
k_m=\begin{cases} 
\lfloor  m\ln m\rfloor &\text{ if $\lfloor  m\ln
 m\rfloor$ is odd};\\
 \lfloor  m\ln m\rfloor+1 &\text{ if $\lfloor  m\ln m\rfloor$ is even}.
\end{cases}
\end{align}
Thus $k_m$ is {\em always} odd; $k_m$, $m=2,3,\dots$ equal $1, 3, 5, 7, 9, 13, 15, 19, 23,$ etc.
Define also 
$$
A_m=\{S(j)=0\text{ for some } i_{k_m}\le j< i_{k_m+1}\},
$$
the event that $S(j)$ hits zero for the steps within  block $B_{k_m}$, and the sequence of sigma-algebras 
$$
\G_m=\F_{i_{k_{m+1}}-1}=
\sigma\left(\bigcup_{\ell=1}^{k_m} \sigma\left(
X^{(\ell)}_1,X^{(\ell)}_2,\dots,X^{(\ell)}_{L_\ell}
\right)
\right).
$$
Intuitively, $\G_m$ contains all the information about the walk during its steps corresponding to the first $k_m$ blocks.

To simplify notations, let us now write $k=k_m$ and $k'=k_{m-1}$, and observe that
\begin{align}\label{BCkk}
k-k'&=k_m-k_{m-1} 
\ge  m\ln m - (m-1)\ln (m-1)-2= \ln m -1+O\left(\frac1{m}\right)\ge \ln m-2
\end{align}
for large $m$.

Let us split $S(j)$ where $j\in [i_k,i_{k+1})$, as follows:
\begin{align*}
S(j)&=S(i_{k'})+\sum_{n=k'+1}^{k-1}\left(X_1^{(n)}+\dots+X_{L_n}^{(n)}\right)
+k\cdot\left(X_1^{(k)}+\dots+ X_{j-i_k}^{(k)}\right)
\\ &=
S(i_{k'})+\left[\sum_{n=k'+1}^{k-2} \left(X_1^{(n)}+\dots+X_{L_n}^{(n)}\right)
 + (k-1)\sum_{\ell=1}^{i_k-2k^2-1}
X_{\ell}^{(k-1)}
\right]\\
& +(k-1)\cdot\Sigma_3+k\cdot\left(X_1^{(k)}+\dots+ X_{j-i_k}^{(k)}\right)
\\ & =\Sigma_1+\Sigma_2+(k-1)\cdot \Sigma_3+k\cdot \Sigma_4
\end{align*} 
where $\Sigma_1=S(i_{k'})$ and
\begin{align*}
\Sigma_2&=\sum_{n=k'+1}^{k-2} n T_n+(k-1)T_{k-1}',\quad T_{k-1}'=\sum_{\ell=i_{k-1}}^{i_k-1-2k^2} X_\ell^{(k-1)};\\
\Sigma_3&=X_{i_k-2k^2}^{(k-1)}+X_{i_k-2k^2+1}^{(k-1)}+\dots+X_{i_k-2}^{(k-1)}+X_{i_k-1}^{(k-1)};\\
\Sigma_4&=X_1^{(k)}+X_2^{(k)}+\dots+ X_{j-i_k}^{(k)}.
\end{align*}
Note that $\Sigma_i$, $i=1,2,3,4$ are independent, and $\Sigma_3$ has precisely $2k^2$ terms.

\vskip 5mm
\noindent\underline{\bf Step 3: Estimating $\Sigma_1$}

Recall that $k=k_m$, $k'=k_{m-1}$ and let
$$
E_{m-1}=\left\{\left|\Sigma_1\right|< k\sqrt { L_k}\right\}\in \G_{m-1}.
$$
By Lemma~\ref{lem:AH} and~\eqref{tcond}, assuming $k$ is large,
\begin{align}\label{eqEm}
\P(E_{m-1}^c)&\le \P(|S(k')|\ge k' \sqrt{L_k})
\le 2\exp\left(-\frac{k'^2\cdot L_k}{2 (L_1+2^2\cdot L_2+3^2\cdot L_3+\dots+k'^2\cdot L_{k'})}\right)
\\
& \le 2\exp\left(-\frac{L_k}{2 (L_1+ L_2+L_3+\dots+L_{k'})}\right)
\le 2\exp\left(-(1+\eps/2)\ln k\right)
=
\frac{2}{k^{1+\eps/2}}=:\eps_m .
\nonumber
\end{align}

\vskip 5mm
\noindent\underline{\bf Step 4: Estimating $\Sigma_2$}

Again, by Lemma~\ref{lem:AH} and~\eqref{tcond},  assuming that $k$  is sufficiently large,
\begin{align*}
\P\left(|\Sigma_2|\ge k  \sqrt{\frac{L_k}r}\right)&\le 
 2\exp\left(-\frac{k^2 r^{-1} L_k}{2[(k'+1)^2L_{k'+1}+\dots+
(k-2)^2 L_{k-2}+ 
 (k-1)^2(L_{k-1}-2k^2)]}\right)
 \\
 &\le 
 2\exp\left(-\frac{r^{-1}\, L_k  }{2[L_{k'+1}+\dots+L_{k-1}-2k^2]}\right) \le 
  2\exp\left(-1\right)
=0.7357588824\dots.
\end{align*}
Consequently, 
\begin{align}\label{eq:beta}
\P\left(|\Sigma_2|< k\sqrt{{L_k}/r}\right)&\ge 0.2\qquad\text{for large $k$}.
\end{align}
\vskip 5mm
\noindent\underline{\bf Step 5: Finishing the proof}

We have a trivial lower bound
\begin{align}\label{eqstar}
\P\left(A_m  \mid E_{m-1},\G_{m-1}\right)
&\ge
 \P\left(A_m \mid |\Sigma_2|< k\sqrt{\frac{L_k}r}, E_{m-1},\G_{m-1}\right) 
 \times \P\left(|\Sigma_2|< k\sqrt{\frac{L_k}r} \mid E_{m-1},\G_{m-1} \right) \nonumber
\\
& =:\mathsf{(*)}\times 0.2\qquad\text{ for large }k
\end{align}
by~\eqref{eq:beta},  since the second multiplier  equals $\P\left(|\Sigma_2|< k \sqrt{L_k/r}\right)$ by independence.

Let
\begin{align*}
\mathsf{Div}_k=\left\{\Sigma_1+\Sigma_2 +(k-1)\Sigma_3 =0 \mod k \right\}
=\left\{\Sigma_1+\Sigma_2 -\Sigma_3=0  \mod k \right\}.
\end{align*}
Since only on the event ${\rm Div}_k$, it is possible that $S(j)=0$ for some $j$ (since the step sizes are~$\pm k$ in the block $B_k$), we conclude that for large $k$
\begin{align}\label{partune}
\mathsf{(*)}
&=
\P\left(A_m \mid \mathsf{Div}_k,  |\Sigma_2|<k\sqrt{\frac{L_k}r}, E_{m-1},\G_{m-1}\right) 
\times 
\P\left( \mathsf{Div}_k \mid  |\Sigma_2|<k\sqrt{\frac{L_k}r}, E_{m-1},\G_{m-1}\right)  \nonumber
\\
& \ge
\P\left(A_m \mid \mathsf{Div}_k,  |\Sigma_2|< k\sqrt{\frac{L_k}r}, E_{m-1},\G_{m-1}\right)  \times \frac{c_1}{2k}
\end{align}
due to the  fact that by Corollary~\ref{cordiv},
$\P\left(\mathsf{Div}_k\mid \F_{i_k-2k^2-1}\right)\ge c_1/(2k).$
On the other hand,
\begin{align}\label{partdeux}
\P\left(A_m \mid \mathsf{Div}_k,  |\Sigma_2|< k\sqrt{\frac{L_k}r}, E_{m-1},\G_{m-1}\right) 
&\ge  
 \min_{z\in Z_k}
 \P(z+T_m=0\text{ for some }m\in[0, L_k])
 \nonumber
\\ &
  \ge \b:=1-\Phi\left(r^{-1/2}+3\right)>0
\end{align}
where $z+T_m$ is a simple random walk starting at $z$
(see Lemma~\ref{lemLD}), and $$Z_k=\left\{z\in\Z:\,|z|\le (r^{-1/2}+3)\, \sqrt{L_k}\right\}.$$
Indeed, using the last part of~\eqref{tcond}, and the conditions we imposed, we have
$$
\left|\Sigma_1+\Sigma_2+(k-1)\Sigma_3\right|\le k\sqrt{L_k}+k  \sqrt{L_k/r}+2(k-1)k^2<(1+r^{-1/2} +2) k\, \sqrt{L_k}
$$ 
for large $k$,  $S(j)=\left[\Sigma_1+\Sigma_2+(k-1)\Sigma_3\right]+k\cdot \Sigma_4$,
and by Lemma~\ref{lem:CLT}
\begin{align*}
\liminf_{k\to\infty}  \min_{z\in Z_k}
 \P(z+T_m=0\text{ for some }m\in[0, L_k])
\ge 2\,
 \P(\eta>r^{-1/2} +3)=2 \b,
\end{align*}
so the minimum in~\eqref{partdeux} is $ \ge \b$ for all sufficiently large $k$.

Finally, from~\eqref{eqstar}, \eqref{partune}, and~\eqref{partdeux} we get that 
$$
\sum_m \P\left(A_m\mid E_{m-1},\G_{m-1}\right)\ge \sum_m \frac{ 0.2\, c_1 \b}{2m\log m}=+\infty
$$
and $\P(E_m^c)$ is summable by~\eqref{eqEm}, so we can apply Lemma~\ref{lem:likeBC} of  Appendix~1 to conclude that events $A_m$ occur infinitely often and thus our ${\bf a}$-walk is recurrent.
\end{proof}

\section{Continuous example}
The example of ${\bf a}$-walk described in Theorem~\ref{th:logint} roughly corresponds  to the case $a_k=\lceil \log_\g k \rceil$, $k=1,2,\dots$. But what if $a_k$'s take non-integer values, but, for example, equal
$$
a_k= \log_\g k\equiv c \ln k,\qquad  k=1,2,\dots,
$$
where $\g=e^{1/c}>1$? In this Section we will study this example. It is unreasonable to assume that such ${\bf a}$-walk is recurrent, because of the irrationality of the step sizes, however, we might want to investigate if this walk is $C$-recurrent for {\em some} $C>0$. Our main result is the following
\begin{thm}\label{th:contlog}
Let $c>0$ and $a_k= c\ln k$. Then the ${\bf a}$-walk is a.s.\ $C$-recurrent for every $C>0$.
\end{thm}
To prove this theorem, it is sufficient to show that whatever the value $c>0$ is, $\{|S(n)|\le 3\}$ happens i.o.\ almost surely. Indeed, take any $C>0$. Then the statement that ${\bf a}'$-walk with $a_k'=\frac{3c}{C}\ln k$, $k=1,2,\dots,$ is $3$-recurrent is equivalent to the statement that ${\bf a}$-walk with $a_k=c\ln k$ is $C$-recurrent.

The proof will proceed similarly to that of Theorem~\ref{th:logint}. Let us define $k_m$ slightly differently from~\eqref{defkm}; namely, let
\begin{align*}
k_m=\begin{cases} 
\lfloor  m\ln m\rfloor &\text{ if $\lfloor  m\ln
 m\rfloor$ is even};\\
 \lfloor  m\ln m\rfloor-1 &\text{ if $\lfloor  m\ln m\rfloor$ is odd}.
\end{cases}
\end{align*}
Thus now $k_m$ are always {\em even}. As before, set $k=k_m$, and $k'=k_{m-1}$, and define
$$
i_k=\lceil \g^k\rceil=\max\{i\ge 1:\ a_i< k\}+1
=\min\{i\ge 1:\ a_i\ge k\}\in [\g^k,\g^k+1),
$$ 
i.e., the first index when $a_i$ starts exceeding $k$.
For $i\in J_k:=[i_k,i_{k+1})$  write
\begin{align}\label{Ssplit}
\begin{array}{lcccc}
S(i)&=S(i_{k'}-1)&+\left[S(i_k-1)-S(i_{k'}-1)- \Sigma_3\right]
&+ \Sigma_3
&+[S(i)-S(i_k-1)]
\\ &=
\Sigma_1&+\Sigma_2 &+ \Sigma_3 &+\Sigma_4(i)
\end{array}
\end{align} 
where 
\begin{align*}
\Sigma_3&=\left[S(i_{k-1}+k^2-1)-S(i_{k-1}-1)\right]
+\left[S(i_k)-S(i_k-k^2)\right].
\end{align*}
Note that $\Sigma_i$, $i=1,2,3,4$ are independent, and $\Sigma_3$ has $2\cdot k^2$ terms, and contains the first $k^2$ and the last $k^2$ steps of the walk, when the step sizes lie in $[k,k+1)$.
\\[5mm]
Let
\begin{align}\label{Edef}
E_{m-1}=\left\{|\Sigma_1|< k \sqrt{i_k}\right\}=\left\{S(i_{k_{m-1}})< k_m \sqrt{i_{k_m} }\right\}
\end{align}
By Lemma~\ref{lem:AH}, since $a_i<k'<k$ for $i< i_{k'}$,
\begin{align}\label{Emcase2}
\begin{split}
\P(E_{m-1}^c)&=\P\left(|\Sigma_1|\ge k \sqrt{i_k}\right) \le  2
\exp\left(-\frac{i_k k^2}{2\sum_{j=1}^{i_{k'}}
a_j^2}\right)
\le
 2\exp\left(-\frac{i_k}{2i_{k'}}\right)
\le 2\exp\left(-\frac{\g^k-1}{2\g^{k'}}\right)
\\
&=2\exp\left(-\frac{\g^{k-k'}(1+o(1))}{2}\right)
=2\exp\left(-\frac{\g^{\ln m -2}}{2+o(1)}\right)
=2\exp\left(-\frac{m^{\ln\g}}{2\g^2+o(1)}\right)=:\eps_{m-1}
\end{split}
\end{align} 
using~\eqref{BCkk} for $k$ sufficiently large\footnote{Note that \eqref{BCkk} was stated for $k_m$ defined slightly differently, however, it holds here as well.}. Observe that $\eps_m$ is summable.
\\[5mm]
Similarly, by  Lemma~\ref{lem:AH} 
\begin{align*}
\P\left(|\Sigma_2|\ge 2k \sqrt{i_k}\right) 
\le 
 2
\exp\left(-\frac{4k^2 i_k}{2k^2\left(i_k-i_{k'}-2k^2\right)}\right)
<2\, e^{-2}=0.27\dots
\end{align*}
Hence,
\begin{align}\label{eqFm} 
\P(F_m)\ge 0.72\quad\text{where }F_m=\left\{|\Sigma_2|< 2k \sqrt{i_k}\right\}.
\end{align}

\begin{lemma}\label{lem:dense}
Let $n=k^2$ where $k$ is an even positive integer, and assume also that $k$ is sufficiently large. Suppose that $X_i$, $Y_i$, $i=1,2,\dots,n$, are i.i.d.\ $\sb$.  Let 
\begin{align}\label{eq:T}
T=(k-1)\,(X_1+X_2+\dots+X_n)+k\,(Y_1+Y_2+\dots+Y_n).
\end{align}
Then
$$
\P(T= j)\ge \frac{c_1^2}{4n}\quad  \text{ for each }j=0,\pm 2,\pm 4,\dots,\pm n
$$
where $c_1$ is the constant from Lemma~\ref{lemLD}.
\end{lemma}
\begin{proof}
It follows from Corollary~\ref{cordiv} that
\begin{align}\label{eqXY}
\P(X_1+\dots+X_n=\ell)\ge \frac{c_1}{2k} ,\quad \P(Y_1+\dots+Y_n=\ell)\ge \frac{c_1}{2k} \quad\text{for all even $\ell$ such that }|\ell|\le 2k.
\end{align}
Let $j= 2\tilde{j}\in\{0,2,4,\dots,n-2,n\}$. Consider the sequence of $k-1$ numbers 
$$
\tilde{j},\ \tilde{j}-k,\ \tilde{j}-2k,\ \tilde{j}-3k,\dots, \tilde{j}-(k-2)k;
$$
they all give different remainders when divided by $k-1$. Hence there must be an $m\in \{0,1,\dots,k-2\}$ such that $\tilde{j}-mk=b(k-1)$ and $b$ is an integer; moreover, since $0\le \tilde{j}\le n/2$, we have $b\in\left[-\frac{k(k-2)}{k-1}, \frac{n}{2(k-1)}\right]$. For such~$m$ and~$b$ we have $j=2\tilde{j}=(2m)k+(2b)(k-1)$, and, since both  $|2m|$ and $|2b|\le 2k$, 
$$
\P(T=j)\ge \P(X_1+\dots+X_n=2b)\cdot\P(Y_1+\dots+Y_n=2m)
\ge \left(\frac{c_1}{2k}\right)^2=\frac{c_1^2}{4n}
$$
by~\eqref{eqXY}. The result for negative $j$ follows by symmetry.
\end{proof}

\begin{corollary}\label{cor:dense}
Let  $\eps=\frac{2ck^4}{\g^{k-1}}$. Then for large even $k$
$$
\P\left(\Sigma_3\in [j-\eps,j+\eps]\right)\ge \frac{c_1^2}{4k^2}\quad  \text{ for each }j=0,\pm 2,\pm 4,\dots,\pm k^2.
$$
\end{corollary}
\begin{proof}
$\Sigma_3$ has the same distribution as
\begin{align*}
\sum_{\ell=1}^{k^2}c\ln (i_{k-1}-1+\ell)\, X_\ell
+\sum_{\ell=1}^{k^2}c\ln (i_k-\ell)\, Y_\ell
\end{align*}
for some i.i.d.\ $X_\ell, Y_\ell\sim\sb$.
At the same time, for $\ell\ge 1$,
\begin{align*}
|c\ln (i_{k-1}-1+\ell)-(k-1)|&=|c\ln (\lceil\g^{k-1} \rceil+ \ell-1)-(k-1)|
\\&\le |c\ln (\g^{k-1} + \ell)-(k-1)|
= c\ln \left(1+ \frac{\ell}{\g^{k-1}}\right)\le \frac{c\ell}{\g^{k-1}} 
\end{align*}
Similarly,
\begin{align*}
k-c\ln (i_k-\ell)&=k- c\ln (\lceil\g^k \rceil-\ell)
=k-c\ln (\g^k -\ell')=-c\ln \left(1-\frac {\ell'}{\g^k}\right)
\in \left[0,\frac{c\ell}{\g^{k-1}}\right]
\end{align*}
for some $\ell'\in[\ell-1,\ell]$, assuming $\ell=o(\g^k)$.
As a result, for $T$ defined by~\eqref{eq:T},
$$
|\Sigma_3-T|\le \sum_{\ell=1}^{k^2}\frac{2c\ell}{\g^{k-1}} 
=\frac{ck^2(k^2+1)}{\g^{k-1}}\le \frac{2ck^4}{\g^{k-1}}.
$$
Now the result follows from Lemma~\ref{lem:dense}.
\end{proof}

\begin{proof}[Proof of Theorem~\ref{th:contlog}]
Recall that $J_k=[i_k,i_{k+1})$ and define
$$
A_m=\{S(i)=0\text{ for some }i\in J_{k_m}\}.
$$
Then
\begin{align}\label{eqstarcont}
\P(A_m\mid E_{m-1},{\cal G}_{m-1})
\ge 
0.72\times \P(A_m\mid F_m, E_{m-1},{\cal G}_{m-1})
\end{align}
(please see the definition of $F_m$ in~\eqref{eqFm}).
Recall formula~\eqref{Ssplit} and write
$$
\tilde S(i)=S(i)-\Sigma_3=\Sigma_1+\Sigma_2+\Sigma_4(i). 
$$
From now on assume that  $|\Sigma_1|< k\sqrt{i_k}$ and $|\Sigma_2|< 2k\sqrt{i_k}$, that is, $E_{m-1}$ and $F_m$ occur. Also assume w.l.o.g.\ that $\Sigma_1+\Sigma_2\ge 0$. 
Let
$$
L_k=i_{k+1}-i_k-k^2=(\g-1)\g^k+o(\g^k).
$$
Consider a simple random walk with steps $Y_i\sim\sb$ during its first $L_k$ steps. The probability that its minimum will be equal to or below the level $-3\sqrt{i_k}=-\frac{3+o(1)}{\sqrt{\g-1}}\,\sqrt{L_k}$ converges by Lemma~\ref{lem:CLT} to
$$
2\,\P\left(\eta>\frac{3}{\sqrt{\g-1}}\right)
=2-2\,\Phi\left(\frac{3}{\sqrt{\g-1}}\right)=:2 c_2\in (0,1)
$$
as $k\to\infty$ (recall that $\eta\sim{\cal N}(0,1)$).
As a result, by Proposition~\ref{prop:Saty}, the probability that  for some $j_0\in \{i_k,i_k+1,i_k+2,\dots,i_k+L_k-1\}$ we have the downcrossing, that is,
$$
\tilde S(j_0-1)\ge 0> \tilde S(j_0)
$$
is bounded below by $c_2$ for $k$ sufficiently large. 
Formally, let 
\begin{align*}
j_0&=\inf\{j>i_k: \ \tilde S(j)<0\};\\
\mathcal{C}_0&=\{i_k\le j_0\le i_k+L_k-1\},
\end{align*}
so we have showed that on $E_{m-1}\cap F_m\cap \{\Sigma_1+\Sigma_2>0\}$ we have $\P(\mathcal{C}_0)>c_2$ for large $k$.

Now assume that event $\mathcal{C}_0$ occurred and define additionally
\begin{align*}
b_0&=\tilde S(j_0)\in (-k-1,0];\\
{\cal C}&=\left\{\max_{0\le h\le k^2} \sum_{g=1}^h X_{j_0+g}\ge k\right\}.
\end{align*}
Again, from Lemma~\ref{lem:CLT},
$$
\P({\cal C})
=2\, \P(X_{j_0+1}+X_{j_0+2}+\dots+X_{j_0+k^2}\ge k)\to 2(1-\Phi(1))=0.3173\dots \quad\text{as }k\to\infty.
$$
From now on assume that $k$ is so large that $\P({\cal C})>0.2$. On the event ${\cal C}$ there exists an increasing  sequence $j_1,j_2,\dots,j_k$ such that
$$
j_0<j_1<j_2<\dots<j_k\le j_0+k^2<i_{k+1}
$$
such that $X_{j_0+1}+X_{j_0+2}+\dots+X_{j_\ell}=\ell$ for each $\ell=1,2,\dots, k$, since the random walk must pass through each integer in $\{1,2,\dots,k\}$ in order to reach level $k$.

For $\ell=1,2,\dots,k$, define
\begin{align*}
b_\ell:=\tilde S(j_\ell)&=b_0+\sum_{h=j_0+1}^{j_\ell} a_h X_h
=b_0+a_{j_0}\left[\sum_{h=j_0+1}^{j_\ell}  X_h\right]
+\sum_{h=j_0+1}^{j_\ell} \left(a_h-a_{j_0}\right) X_h
\\
&=b_0+a_{j_0}\ell +O\left(\frac{k^4}{\g^k}\right)
\end{align*}
since for $h\in [j_0,j_0+k^2]\subseteq [i_k,i_{k+1})$ we have
$$
|a_h-a_{j_0}|=c\left | \ln \frac{a_h}{a_{j_0}}\right|\le c\left | \ln \frac{i_k+k^2}{j_k}\right|=O\left(\frac{k^2}{\g^k}\right).
$$
As a result,
$$
-(k+1)<b_0<b_1<b_2<\dots<b_{k-1}<(k-1)(k+1)<k^2
$$
and moreover the distance between consecutive $b_g$'s is at least two (provided $k$ is large). 
For $\ell=0,1,\dots,k-1$ define
$$
\tilde b_\ell=\sup\left\{x\in 2\Z: x+b_\ell\in \left[\frac12,3-\frac12\right)\right\}
\equiv-2\left\lfloor\frac{b_\ell}2-\frac14\right\rfloor.
$$
Then $\tilde b_\ell$'s are  all distinct even integers satisfying   $|\tilde b_\ell|\le k^2$.

As a result,
\begin{align*}
\P(A_m&\mid F_m, E_{m-1},{\cal G}_{m-1})\ge c_2 \times 0.2 \times 
\P(S(i)\in [0,3]\text{ for some }i\in J_k\mid \mathcal{C},\mathcal{C}_0)
\\
&\ge c_2 \times 0.2 \times 
\P(\tilde S(i_\ell)+\Sigma_3\in [0,3]\text{ for some }\ell=0,1,\dots,k-1\mid \mathcal{C},\mathcal{C}_0)
\\
&= c_2 \times 0.2 \times 
\P(b_\ell+\Sigma_3\in [0,3]\text{ for some }\ell=0,1,\dots,k-1\mid \mathcal{C},\mathcal{C}_0)
\\
&\ge c_2 \times 0.2 \times 
\P(\tilde \Sigma_3\in [\tilde b_\ell-\eps,\tilde b_\ell+\eps]\text{ for some }\ell=0,1,\dots,k-1\mid \mathcal{C},\mathcal{C}_0)\\
&= c_2 \times 0.2 \times 
\sum_{\ell=0}^{k-1}\P( \Sigma_3\in [\tilde b_\ell-\eps,\tilde b_\ell+\eps]\mid \mathcal{C},\mathcal{C}_0)
\ge c_2 \times 0.2 \times k\times \frac{c_1^2}{4k^2}
=\frac{c_1^2 c_2}{20k}
\end{align*}
assuming that $\eps$ in Corollary~\ref{cor:dense} is sufficiently small.
Finally,
\begin{align*}
\P(A_m\mid E_{m-1},{\cal G}_{m-1})
\ge 0.72\times \P(A_m\mid F_m, E_{m-1},{\cal G}_{m-1})\ge  \frac{0.72\, c_1^2 c_2}{20k_m}
\ge \frac{0.036\, c_1^2 c_2}{ m\ln m}
\end{align*}
the sum of which diverges.
Hence, recalling~\eqref{Emcase2}, we can again apply Lemma~\ref{lem:likeBC}.
\end{proof}

\section{Sublinear  growth of step sizes}
Throughout this Section we assume
$$
a_k=\lfloor k^\b\rfloor,\qquad  0<\b<1.
$$

\begin{prop}\label{prop:SUL}
Let $S(n)=a_1 X_1+\dots+a_n X_n$ where $a_k=\lfloor k^\b\rfloor$, $0<\b<1$. Then 
$$
\P(|S(n)|=z) \le \frac{\nu}{n^{1/2+\b}}
\qquad\text{for all large }n,
$$
for some  $\nu>0$.
\end{prop}
\begin{proof} 
Let $F_n(t)=\prod_{k=1}^n |\cos(t a_k)|$.
For all $z\in\Z$  we have
$$
\P(S(n)=z)=\frac1{2\pi}\int_{-\pi}^\pi
e^{-itz}\E e^{it S(n)}\dd t
\le \frac1{2\pi}\int_{-\pi}^\pi
\left|\E e^{it S(n)}\right|\dd t
=\frac1{2\pi}\int_{-\pi}^{\pi}
F_n(t)\dd t
\le\frac{\nu}{n^{1/2+\b}}
$$
by Lemma~\ref{lemSUL} for some $\nu>0$, for all large $n$. 
\end{proof}

\begin{thm}\label{thm:beta}
Suppose that $a_k=\lfloor k^\b\rfloor$, $0<\b<1$. Then the ${\bf a}$-walk is a.s.\ transient.
\end{thm}
\begin{proof}
In the case $\b>1/2$ the a.s.\ transience follows immediately from Borel-Cantelli lemma and Proposition~\ref{prop:SUL}, as $\sum_n \P(|S(n)|\le C)<\infty$ for each $C\ge 1$.

Assume from now on that $0<\b\le 1/2$.
Define $k_m$, $\Delta_m$, $m_n$ as in Case 3 of the proof of Lemma~\ref{lemSUL}.  Fix a positive integer $M$ and consider now only those $n$ for which $m_n=M$.  Let
$I_M=\{k_M,k_M+1,\dots,k_{M+1}-1\}$.
Note that the elements of $I_M$ are precisely those $n$ for which $a_n=M$, and  that the caridnality of $I_M$ is of order $M^{1/\b-1}$. Next, fix some $z\in\Z$ and define
$$
E_M=E_M(z)=\{S(n)=z\text{ for some } n\in I_M\}\}.
$$
For each $z$ we will show that $\sum_M \P(E_M)<\infty$, and so by the Borel-Cantelli lemma a.s.\ only finitely many events $E_M$ occur. Since $S(n)$ takes only integer values, this will imply that the walk is not $C$-recurrent  for any $C\ge 0$.

So, fix $z$ from now on, write $S(n)=S(k_M)+R(n)$ where
$$
R(n)=\sum_{i=k_M}^n  a_{i} X_i=M\,\sum_{i=k_M}^n   X_i,
$$
Observe also that $S(k_M)$ and $R(n)$ are independent.
In order $S(n)=z$ for some $n\in I_M$, we need that $S(k_M)=z\mod M$. Let $Q=M^{\frac 1{2\b}+1-\eps}$ for an $\eps\in(0,1/2)$. Assuming $M$ is so large that $Q\ge 2|z|$,
\begin{align*}
\P(|R(n)|\ge Q-|z|)&\le 2\,\exp\left(-\frac{(Q/2)^2}{2 M^2\cdot (k_{M+1}-k_M)}\right)=
2\,\exp\left(-\frac{\b Q^2}{(8+o(1)) M^2\cdot M^{\frac1{\b}-1}}\right)\\
&= 2\,\exp\left(-\frac{\b M^{1-2\eps}}{8+o(1)} \right)
\quad\text{for all }n\in I_M
\end{align*}
by Lemma~\ref{lem:AH}; hence 
$$
\P(\max_{n\in I_M}|R(n)|\ge Q-|z|)\le |I_M|\times 
2\,\exp\left(-\frac{\b M^{1-2\eps}}{8+o(1)}\right)
=:\alpha_M=O\left(M^{\frac1\b-1} e^{-\frac{\b M^{1-2\eps}}{8+o(1)}}\right),
$$
which is summable in $M$. So, 
\begin{align*}
\P(E_M)&\le \P\left(E_M,\ \max_{n\in I_M}|R(n)|< Q-|z|\right)+\P\left(\max_{n\in I_M}|R(n)|\ge Q-|z|\right)
\\ &
=\P\left(E_M,\ S(k_M)=z\mod M,\ \max_{n\in I_M}|R(n)|< Q-|z|\right)+\alpha_M=(*)+\alpha_M
\end{align*}
where the term $\alpha_M$ is summable since $1-2\eps>0$.
Since $E_M$ implies  implies $-S(k_M)=R(n)-z$ for some $n\in I_M$,
\begin{align*}
(*)&\le
\P(|S(k_M)|<Q,\ S(k_M)=z\mod M)
=\sum_{j:\ |j|<Q,\ j=z\mod M}
\P(|S(k_M)|=j)
\\ &
\le \frac{\nu}{k_M^{1/2+\b}}\times |\{{j:\ |j|<Q,\ j=z\mod M}\}|
\le \frac{\nu+o(1)}{M^{1+\frac1{2\b}}}\times\left[ \frac{2Q+1}{M}+1\right]
=\frac{(\nu+o(1))}{\pi\, M^{1+\eps}}
\end{align*}
by Proposition~\ref{prop:SUL}. The RHS is summable in $M$,  which concludes the proof.
\end{proof}

\begin{rema}
By setting $\eps=1/2-\d/2$, where $\d>0$ is very close to zero, we can ensure that 
\begin{align*}
\P(|S(n)|< M^{1/2-\d} \text{ for some }n\in I_M)
&\le \sum_{z:|z|<M^{1/2-\d}}\P(E_M(z))\le
 \left[\frac{\mathrm{const}}{M^{1+\eps}} +\alpha_M\right]\times 2M^{1/2-\d}
\\ & =\frac{2\,\mathrm{const}}{M^{1+\d/2}} +2M^{1/2-\d}\alpha_M
\end{align*}
is summable. Hence, a.s.\ eventually $|S(n)|$ will be larger that $n^{\b/2-\d}$ for any $\d>0$.
\end{rema}

\section*{Appendix 1: Modified conditional Borel-Cantelli lemma}
\begin{lemma}\label{lem:likeBC}
Suppose that we have an increasing sequence of sigma-algebras~$\G_m$ and a sequence of $\G_m$-measurable events $A_m$ and $E_m$  such that
\begin{align*}
\P(A_m\mid E_{m-1},{\cal G}_{m-1})\ge \a_m,\quad
\P(E_m^c)\le \eps_m\quad\text{a.s.}
\end{align*}
where the non-negative $\a_n$ and $\eps_n$ satisfy
\begin{align}\label{BCsum}
\sum_m \a_m=\infty,\quad \sum_m \eps_m<\infty.
\end{align}
Then $\P(A_m\text{ i.o.})=1$.
\end{lemma}
\begin{proof}
Let $m>\ell\ge 1$ and $B_{\ell,m}=\bigcap_{i=\ell}^m A_i^c$.
We need to show that for any $\ell\ge 1$
\begin{align}\label{BCrez}
\P(B_{\ell,\infty})=\P\left(A_{\ell}^c \cap A_{\ell+1}^c\cap A_{\ell+2}^c\cap \dots \right)=0.
\end{align}
We have for $m\ge \ell+1$
\begin{align}\label{BCit}
\begin{split}
\P(B_{\ell,m})&=\P\left(A_m^c\cap B_{\ell,m-1} \right)\le 
\P\left(A_m^c\cap B_{\ell,m-1} \cap E_{m-1}\right)+\P\left(E_{m-1}^c\right)\\
&=
\P\left(A_m^c\mid B_{\ell,m-1} \cap E_{m-1}\right)
\P\left(B_{\ell,n-1} \cap E_{m-1}\right)
+\P\left(E_{m-1}^c\right)\\
&\le
\P\left(A_m^c\mid B_{\ell,m-1} \cap E_{m-1}\right)
\P\left(B_{\ell,m-1} \right)
+\eps_{m-1}
 \le (1-\a_m)\P\left(B_{\ell,m-1}\right)
+\eps_{m-1}.
\end{split}
\end{align}
By induction over $m,m-1,m-2,\dots,\ell+1$ in~\eqref{BCit}, we get that
\begin{align*}
\P(B_{\ell,m})&\le \eps_{m-1}+
(1-\a_m)\left[(1-\a_{m-1})\P\left(B_{\ell,m-2} \mid {\cal G}_m\right)+\eps_{m-2}\right]\le \dots
\\
& \le \eps_{m-1}+(1-\a_m)\eps_{m-2}
+(1-\a_m)(1-\a_{m-1})\eps_{m-3}+
\dots
\\ &
+(1-\a_m)(1-\a_{m-1})\dots(1-\a_{\ell+2})\eps_{\ell}
+(1-\a_m)(1-\a_{m-1})\dots (1-\a_{\ell+1}).
\end{align*}
Hence, for any integer $M\in (\ell, m)$
\begin{align*}
\P(B_{\ell,m})&\le \eps_{m-1}+\eps_{m-2}+\dots+\eps_{M}
\\ &
+(1-\a_m)(1-\a_{m-1})\cdots(1-\a_{M+1})\eps_{M-1}
+\dots+(1-a_m)(1-\a_{m-1})\dots(1-\a_{\ell+2})\eps_{\ell}
\\ &
+(1-\a_m)(1-\a_{m-1})\dots (1-\a_{\ell+1})
\\ &
\le [\eps_{m-1}+\eps_{m-2}+\dots+\eps_{M}]
\\ &+
(1-\a_m)(1-\a_{m-1})\cdots(1-\a_{M+1})
[\eps_{M-1}+\eps_{M-2}+\dots+\eps_\ell+1]
\end{align*}
Fix any $\delta>0$. By~\eqref{BCsum} we can find an
$M$ be so large that $\sum_{i=M}^\infty \eps_i<\d/2$.
Then, again by~\eqref{BCsum}, there exists an $m_0>M$ such that  $\displaystyle \prod_{i=M+1}^{m_0} (1-\a_i)<\frac{\delta}{2\left(1+\sum_{i=\ell}^{M-1} \eps_i\right)}$. Hence, for all $m\ge m_0$ we have $\P(B_{\ell,m})\le \d/2+\d/2=\d$. Since $\d>0$ is arbitrary, and $B_{\ell,m}$ is a decreasing sequence of events in~$m$, we conclude that $\P(B_{\ell,\infty})=0$, as required.
\end{proof}

\section*{Appendix 2: Generalization of Blair Sullivan's  results}

Let $a_k=\lfloor k^\b\rfloor$, where $0<\b<1$.

\begin{lemma}\label{lemSUL}
Let $F_n(t)=\prod_{k=1}^n |\cos(t a_k)|$. Then 
$$
\int_{-\pi}^\pi F_n(t)\dd t= \frac {\sqrt{8\pi(1+2\b)}+o(1)}{n^{\b+1/2}}
\quad\text{ as }n\to\infty.
$$
\end{lemma}
\begin{rema}
Note that for $\beta=1$ we would have obtained the same result as in~\cite{SUL}.
\end{rema}
\begin{proof}
We will proceed in the spirit of~\cite{SUL}.
Note that by symmetry
$$
\int_{-\pi}^\pi F_n(t)\dd t= 2 \int_0^{\pi} F_n(t)\dd t
=2 \int_0^{\pi/2} F_n(t)\dd t+2 \int_0^{\pi/2} F_n(\pi-t)\dd t
=4 \int_0^{\pi/2} F_n(t)\dd t
$$
since $|\cos((\pi-t)a_k)|=|\cos(\pi a_k-t a_k)|=|\cos(t a_k)|$ as $a_k$ is an integer.
Let $\eps>0$ be very small and define
\begin{align*}
I_0&=\left[0,\frac1{n^{\b+1/2-\eps}}\right],\quad
I_1=\left[\frac1{n^{\b+1/2-\eps}},\frac{1}{n^\b}\right],\quad
I_2=\left[\frac{1}{n^\b},\frac{c_1}{n^\b}\right],\quad
I_3=\left[\frac{c_1}{n^\b},\frac\pi2\right].
\end{align*}
for some $c_1>1$ to be determined later.
Then 
$$
\int_0^{\pi/2} F_n(t)\dd t=\int_{I_0} F_n(t)\dd t+\int_{I_1} F_n(t)\dd t+\int_{I_2} F_n(t)\dd t+\int_{I_3} F_n(t)\dd t.
$$
We will show that the contribution of all the integrals,  except the first one, is negligible, and estimate the value of the first one.

First, observe that when $0\le t a_k\le \pi/2$ for all $k\le n$, by the elementary inequality $|\cos u|\le e^{-u^2/2}$ valid for $|u|\le \pi/2$, we have
\begin{align}\label{eqint}
F_n(t)&\le\prod_{k=1}^n \exp\left(-\frac{t^2 a_k^2}{2}\right)
= \exp\left(-\frac{t^2}2 \sum_{k=1}^n a_k^2\right)
=\exp\left(-\frac{t^2\, n^{2\b+1}(1+o(1))}{2(1+2\b)} \right)
\end{align}
since $a_k^2=k^{2\b}(1+o(1))$.

\vskip 5mm
\noindent\underline{Case 0}: $t\in I_0$

Here $t a_k\le  \frac1{n^{1/2-\eps}}\ll 1$, hence for $n$ large enough
\begin{align*}
\frac{(ta_k)^2}{2}\le -\ln \cos(t a_k)
=\frac{(ta_k)^2}{2}
+O\left((ta_k)^4\right)
\le (1+o(1))\frac{(ta_k)^2}{2}
\end{align*}
yielding $F_n(t)=\exp\left(-\frac{t^2\, n^{2\b+1} \rho_{n,t}}{2(1+2\b)} \right)
$
where $\rho_{n,t}=1+o(1)$ for large $n$ (compare with~\eqref{eqint}). Since for any $r>0$ we have
\begin{align*}
\int_0^{n^{-\b-1/2+\eps}} 
\exp\left(-\frac{t^2\, n^{2\b+1} r}{2(1+2\b)} \right)
\dd t&=
\frac1{n^{1/2+\b}}\int_0^{n^{\eps}} 
\exp\left(-\frac{s^2\,r}{2(1+2\b)} \right)
\dd s
\\ &=\frac1{n^{1/2+\b}}
\left[\sqrt{\frac{\pi(1 + 2\b)}{2r}}+o(1)\right]
\end{align*}
where the main term is monotone in $r$, by substituting $r=\rho_{n,t}=1+o(1)$ we conclude that
$$
\int_{I_0} F_n(t)\dd t=
\frac1{n^{1/2+\b}}
\left[\sqrt{\pi(1/2 + \b)}+o(1)\right].
$$

\vskip 5mm\noindent\underline{Case 1}: $t\in I_1$

Since $t a_k\le 1<\pi/2$, by~\eqref{eqint} for some $C_2>0$ we have $F_n(t) \le \exp\left(-\frac{n^{2\eps}(1+o(1))}{2(1+2\b)} \right)\le e^{-C_2 n^{2\eps}}$, 
so  $\int_{I_1} F_n(t)\dd t 
\le e^{-C_2 n^{2\eps}}$, which decays faster than polynomially.

\vskip 5mm
\noindent\underline{Case 2}: $t\in I_2$

As in Case 2 in~\cite{SUL}, we will use monotonicity of $F_n(t)$ in $n$. Let $r=c_1^{-1/\b}\in (0,1)$ then $\lfloor rn\rfloor^\b\le (rn)^{\b}=\frac{n^\b}{c_1}$, consequently by~\eqref{eqint}, since $t\le \frac{c_1}{n^\b}$,
$$
F_{\lfloor rn\rfloor}(t)\le \exp\left(-\frac{t^2\, (rn)^{2\b+1}(1+o(1))}{2(1+2\b)} \right)
$$
and since $F_n(t)\le F_{\lfloor rn\rfloor}(t)$, we get a similar bound as in Case 1.

\vskip 5mm
\noindent\underline{Case 3}: $t\in I_3$

Let
\begin{align*}
k_m&=\inf\{k\in\Z_+:\ k^\b\ge m\}=\lceil m^{1/\b}\rceil, \quad m=1,2,\dots;
\\
\Delta_m&=k_{m+1}-k_m=\b^{-1}\, m^\g +O(m^{1/\b-2})  +\rho_0, \qquad \g:=\frac{1-\b}{\b},
\end{align*}
where $|\rho_0|\le 1$.
Then 
$$
a_k=m\quad\text{if and only if}\quad k\in\{k_m,k_m+1,\dots,k_m+\Delta_m-1\, (\equiv k_{m+1}-1)\}.
$$
For $n\in\Z_+$ let 
$$
m_n=\max\{m:\ k_m\le n\}=n^{\b}\, (1+o(1)),
\qquad n\in [k_{m_n},k_{m_n}+\Delta_{m_n}-1].
$$
By the inequality between the mean geometric and the mean arithmetic,
\begin{align*}
F_n(t)&=\sqrt{\prod_{k=1}^n \cos^2(t a_k)}
= \left(\sqrt[n]{\prod_{k=1}^n \cos^2(t a_k)}\right)^{n/2}
\le
\left(\frac{\sum_{k=1}^n \cos^2(t a_k)}n\right)^{n/2}
\\
&=
\left(\frac 12 +\frac{U_n(t)}{2n}\right)^{n/2}
\qquad \text{where }U_n(t)=\sum_{k=1}^n \cos(2 t a_k).
\end{align*}
We will show that if $t$ is not too small, then for some $0\le c<1$ we have $U_n(t)\le c n$ and hence $F_n(t)\le \left(\frac{1+c}2\right)^{n/2}$.
In order to do that, first note that 
$$
U_n(t)\le \sum_{k=1}^{k_{m_n}} \cos(2 t a_k)
+(n-k_{m_n})=
\sum_{m=1}^{m_n} \Delta_m \cos(2 t m)
+(n-k_{m_n}).
$$
Let $r\in (0,1)$ and assume w.l.o.g.\ that $r m_n$ is an integer. For $m\in [r m_n+1,m_n]$ we have 
$$
A\le \Delta_m \le \bar A\text{  where }A=\b^{-1}  (r m_n)^\g+O(1), \quad \bar A=\b^{-1}  m_n^\g+O(1).
$$
Consequently,
\begin{align*}
& \sum_{m=r m_n+1}^{m_n} \Delta_m \cos(2 t m)
\le \sum_{m=r m_n+1}^{m_n} \left[\bar A \cdot{\bf 1}_{\cos(2 t m)\ge 0}+A \cdot {\bf 1}_{\cos(2 t m)<0}
\right] \cos(2 t m)
\\
&=\sum_{m=r m_n+1}^{m_n} \left[\bar A-A\right] \cos(2 t m)\, {\bf 1}_{\cos(2 t m)\ge 0}
+\sum_{m=r m_n+1}^{m_n}A \cos(2 t m)
\\
&\le  (1-r)m_n\left(\bar A-A\right)
+A \sum_{m=r m_n+1}^{m_n} \cos(2tm)
=(1-r)m_n \left(\bar A-A\right)
\\ &  
+ A \left(
 \cos^2\left(r t m_n+t\right) -\cos^2(t m_n + t) + \frac{\cos t}{2\sin t} \left[ \sin(2 t (m_n+ 1)) - \sin(2t(r m_n+1))\right]
 \right)
 \\
&\le \frac{m_n^{1/\b}}{\b} (1-r) \left[1-r^\g+O(m_n^{-\g})\right]+
 m_n^{\g}\left(\frac{r^{\g}}\b+O(m_n^{-\g})\right)\left(1+\frac{1}{|\sin t|}\right)
\end{align*}
Hence, since $m_n^{1/\b}=n+o(n)$,
\begin{align*}
U_n(t)&\le \sum_{m=1}^{r m_n} \Delta_m+\sum_{m=r m_n+1}^{m_n}
 \Delta_m\cos(2tm)+(n-k_{m_n})
=k_{r m_n+1}+\sum_{m=r m_n+1}^{m_n}\Delta_m\cos(2tm)+O(\Delta_{m_n})\\
&\le r^{1/\b}n +\frac{m_n^{1/\b}}{\beta}
(1-r)\left[1-r^{\g}+O(m_n^{-\g})\right]+
 m_n^{\g}\left(\frac{r^{\g}}\b+O(m_n^{-\g})\right)\left(1+\frac{1}{|\sin t|}\right)
 +O(m_n^{\g\b})
 \\
&\le 
n\left(r^{1/\b}+\frac {(1-r)(1-r^{\g})}{\b}\right)+
 \frac{4 n^{1-\b}\b^{-1} r^{\g}}{|\sin t|}+o(n)
\end{align*}
Consider now the the function
$$
h(r,\b):=r^{1/\b}+\frac {(1-r)(1-r^\g)}{\b}
=r^{1/\b}+\frac {(1-r)\left(1-r^{1/\b-1}\right)}{\b}
$$
and note that
$$
h(1-\b,\b)=(1-\b)^{1/\b}+1-(1-\b)^{1/\b-1}
=1-\b(1-\b)^{1/\b-1}\le 1-e^{-1}\b<1-\b/3
$$
since $\sup_{\b\in(0,1)}(1-\b)^{1/\b-1}=e^{-1}$ by elementary calculus.
So if we set $r=1-\b\in (0,1)$, by noting $\ t\le 2 \sin t\ $ for $t\in [ 0,\pi/2]$, we conclude that $U_n(t)\le \left(1-\frac\b4\right) n$ provided that $t\ge \frac{c_1}{n^\b}$ for some $c_1>0$. Consequently,
$\int_{I_3} F_n(t)\dd t \le \left(1-\frac\b8\right)^{n/2}$ for large $n$, which converges to zero exponentially fast.
\end{proof}

\section*{Acknowledgment} We thank Edward Crane and Andrew Wade for providing us with useful references. The research is partially supported by Swedish Science Foundation grants VR 2019-04173 and Crafoord foundation grant no.~20190667.

\end{document}